\documentclass[11pt, reqno]{amsart}

\usepackage{amsmath,amssymb,latexsym}
\usepackage{color}
\usepackage{xcolor}
\usepackage{graphicx}
\usepackage{cite}
\usepackage[colorlinks=true]{hyperref}
\hypersetup{urlcolor=blue, citecolor=red, linkcolor=blue}

\newtheorem{theorem}{Theorem}[section]
\newtheorem{lemma}{Lemma}
\newtheorem{proposition}{Proposition}
\newtheorem{remark}{Remark}

\newtheorem*{main-theorem}{Main Theorem}
\newtheorem*{remark*}{Remark}
\newtheorem*{lemma*}{Lemma A.1}

\numberwithin{equation}{section}

\begin{document}
\title[Two-phase flow model with vacuum]{Non-existence of classical solutions to a two-phase flow model with vacuum}
	
\author{Hai-Liang Li}
		
\address{School of Mathematics and CIT, Capital Normal University, Beijing 100048, People’s Republic of China.}
\email{hailiang.li.math@gmail.com}

\author{Yuexun Wang}

\address{School of Mathematics and Statistics, Lanzhou University, Lanzhou 730000, People’s Republic of China.}
\email{yuexunwang@lzu.edu.cn}

\author{Yue Zhang}
	
\address{School of Mathematics and CIT, Capital Normal University, Beijing 100048, People’s Republic of China.}
\email{yuezhangmath@126.com}

\begin{abstract}
In this paper, we study the well-posedness of classical solutions to a two-phase flow model consisting of the pressureless Euler equations coupled with the isentropic compressible Navier-Stokes equations via a drag forcing term. We consider the case that the fluid densities may contain a vacuum, and the viscosities are density-dependent functions. Under suitable assumptions on the initial data, we show that the finite-energy (i.e., in the inhomogeneous Sobolev space) classical solutions to the Cauchy problem of this coupled system do not exist for any small time. 
\end{abstract}
	
	

\maketitle

\section{Introduction}
This paper concerns the pressureless Euler-Navier-Stokes (Euler-NS) system for a two-phase flow, which is given by
\begin{equation}\label{main0}
	\left\{
	\begin{aligned}
	& \partial_{t}\rho+\text{div} (\rho u)=0,\\
	& \partial_{t}(\rho u)+\text{div} (\rho u\otimes u)=-\rho (u-w),\\
	&\partial_{t}n+\text{div} (nw)=0,\\
	&\partial_{t}(nw)+\text{div} ( n w\otimes w)+\nabla P=\text{div}\tau+\rho(u-w),
	\end{aligned}
	\right.
\end{equation}
where $(x,t)\in\mathbb{R}^d\times[0,\infty)$. Here $\rho=\rho(x,t)$ and $u=u(x,t)$ represent the scalar density and the velocity field of the pressureless Euler equations, while $n=n(x,t)$, $w=w(x,t)$, and $P=P(n)$ represent the scalar density, the velocity field, and the pressure potential of the isentropic compressible Navier-Stokes equations. The equation of state takes the form
\begin{equation}\label{state}
	P(n)=A n^\gamma,\quad \gamma>1. 
\end{equation}
We will set the constant $A$ to be the unit for simplicity.  The viscous stress tensor $\tau$ is defined by
\begin{equation}
\tau=\nu(n)(\nabla w+(\nabla w)^\top)+\lambda(n)\text{div} w \mathbb{I}_d,
\end{equation}
where the coefficients $\nu=\nu(n)$ and $\lambda=\lambda(n)$ are density-dependent functions satisfying
\begin{equation*}
	\nu(n)>0,\quad 2\nu(n)+d\cdot\lambda(n)\geq 0. 
\end{equation*}
The initial data is posed by
\begin{equation}\label{initial}
	(\rho,u,n,w)(x,0)=(\rho_0,u_0,n_0,w_0)(x).
\end{equation}

It is an important issue to study the well-posedness of solutions to the compressible Navier-Stokes (CNS) system (i.e., equations $\eqref{main0}_3$-$\eqref{main0}_4$ without the source term), and there is a huge amount of literature concerning this. 
For small initial data that the density is far away from the vacuum, Serrin \cite{MR0106646} and Nash \cite{MR0149094} obtained the local existence of classical solutions to the CNS system.  The global existence of unique strong solutions was proved by Matsumura-Nishida \cite{MR0555060,MR0564670,MR0713680} in the energy space 
\begin{equation}\label{s1}
	\left\{
	\begin{aligned}
	&n-\bar{n}\in C\big(0,T;H^3(\mathbb{R}^3)\big)\cap C^1\big(0,T;H^2(\mathbb{R}^3)\big), \\
	&w\in C\big(0,T;H^3(\mathbb{R}^3)\big)\cap C^1\big(0,T;H^1(\mathbb{R}^3)\big),
	\end{aligned}
	\right.
\end{equation}
with $\bar{n}>0$ for any $T\in(0,\infty]$, where the additional assumption of small oscillation is required on the perturbation of initial data near the non-vacuum equilibrium state $(\bar{n},0)$. The global well-posedness of strong solutions to the Cauchy problem of the CNS system in the critical Besov type space was established by Danchin \cite{MR1779621}, Chen-Miao-Zhang \cite{MR2675485}, and so on. 
For general initial data containing a vacuum, Cho-Kim \cite{MR2094425,MR2289539}  obtained the local existence of unique strong solutions for the CNS system in the energy space (homogeneous Sobolev space)
\begin{equation*}
	\left\{
	\begin{aligned}
	&n\in C\big(0,T;H^3(\mathbb{R}^3)\big)\cap C^1\big(0,T;H^2(\mathbb{R}^3)\big), \\
	&w \in C\big(0,T;D^3(\mathbb{R}^3)\big)\cap L^2\big(0,T;D^{4}(\mathbb{R}^3)\big),
	\end{aligned}
	\right.
\end{equation*}
where $D^k(\mathbb{R}^3)=\{f\in L_{\textrm{{loc}}}^1(\mathbb{R}^3):\ \nabla f\in H^{k-1}(\mathbb{R}^3)\,\}$, under some additional compatibility conditions on the initial data:
\begin{equation*}
	\left\{ 
	\begin{aligned}
	&-\nu\Delta w_0-(\nu+\lambda)\nabla\textrm{div}w_0+\nabla P_0=n_0g,\\
	&g\in D^1(\mathbb{R}^3),\quad \sqrt{n_0}g\in L^2(\mathbb{R}^3).
	\end{aligned}
	\right.
\end{equation*}
The global existence of classical solutions with small energy but large oscillations to the Cauchy problem was established by Huang-Li-Xin \cite{MR2877344} and Wen-Zhu \cite{MR3597161}.

Comparing \cite{MR0555060,MR0564670,MR0713680} with \cite{MR2094425,MR2289539, MR2877344,MR3597161}, it is nature to ask the question whether the CNS system is well-posed in the inhomogeneous Sobolev space when initial data contains a vacuum, in other words, is similar result valid in the energy space \eqref{s1} when $\bar{n}=0$? Generally, the answer is negative. Indeed, Xin~\cite{MR1488513} first proved that it is impossible to obtain the global existence in the inhomogeneous Sobolev space to the Cauchy problem of the CNS system with compactly supported density. Furthermore, Li-Wang-Xin \cite{MR3925527} proved that the classical solutions do not exist in the inhomogeneous Sobolev space for arbitrary small time if the density has compact support satisfying some nature assumptions. One can infer that the velocity blows up in $L^2$ sense for any small time. These results together show that, in the presence of vacuum, the homogeneous Sobolev space is crucial to study the well-posedness of the Cauchy problem for the CNS system. 
Very recently, Li-Xin \cite{MR4039142,MR4491875,MR4646968} obtained the local and global existence of entropy-bounded solutions to the Cauchy problem of the full CNS system, which indicates that the regularity of the fluid velocity can be propagated in the inhomogeneous space, if the initial density decays to vacuum at the far field with a slow decay rate (no faster than $\mathcal{O}(|x|^{-2})$). They also showed 
that the specific entropy becomes not uniformly bounded immediately after the initial time, as long as the initial density decays to
vacuum at the far field with a fast decay rate (no slower than $\mathcal{O}(|x|^{-4})$) in \cite{Li-Xin 2023}. 
In addition, with a suitable decaying rate of the initial density across the vacuum boundary, the immediate blowup of the entropy-bounded classical solutions to the free boundary problem for the full CNS system was investigated by Liu-Yuan \cite{MR4586591}. 
Tang \cite{MR4409181} considered the CNS-Korteweg system and proved some non-existence results in the inhomogeneous Sobolev space.

The investigation of two-phase flows has been the focus of much attention in the last decade, with numerous significant progress on global weak and strong solutions of the related models.  More precisely, for the generic viscous compressible two-fluid model
\begin{equation}\label{generic tfm}
	\left\{
	\begin{aligned}
	&\alpha^{+}+\alpha^{-}=1,\\
	& \partial_{t}(\alpha^{\pm}n^{\pm})+\text{div} (\alpha^{\pm}n^{\pm} w^{\pm})=0,\\
	& \partial_{t}(\alpha^{\pm}n^{\pm} w^{\pm})+\text{div} (\alpha^{\pm}n^{\pm} w^{\pm}\otimes w^{\pm})+\alpha^{\pm}\nabla P^{\pm}(n^{\pm})=\text{div}(\alpha^{\pm}\tau^{\pm}),
	\end{aligned}
	\right.
\end{equation}
the global weak solution was obtained by Bresch et al. \cite{MR2609956} in the periodic domain and by Bresch-Huang-Li \cite{MR2885607} in one dimensional case, provided that $P^{+}(n^{+})=P^{-}(n^{-})$. When the two pressure functions are unequal, Evje-Wen-Zhu \cite{MR3509002} established the global exsitence and optimal time decay rates of strong solutions near the constant equilibrium state. 
As for a simplified version of \eqref{generic tfm}, the so-called drift-flux model as
\begin{equation}\label{drift-flux}
	\left\{
	\begin{aligned}
	& \partial_{t}\rho+\text{div} (\rho w)=0,\\
	&\partial_{t}n+\text{div} (nw)=0,\\
	&\partial_{t}((\rho+n)w)+\text{div} ( (\rho+n) w\otimes w)+\nabla P(\rho,n)=\text{div}\tau,
	\end{aligned}
	\right.
\end{equation}
Evje-Karlsen \cite{MR2455781} first proved the existence of global weak solutions in a setting where the transition to single-phase flow cannot occur, which implies that the initial vacuum is not allowed. When the initial data close to a constant stable equilibrium, the global well-posedness of strong solutions to the Cauchy problem for \eqref{drift-flux} was established by Guo-Yang-Yao \cite{MR2867815} in Sobolev space and by Hao-Li \cite{MR2982713} in the homogeneous Besov space. 
Besides, Novotn\'{y}-Pokorn\'{y} \cite{MR4062479} released the non-vacuum condition by a severe restriction between the two initial densities that
\begin{equation}\label{ulbound}
0\leq\underline{c}\rho_0\leq n_0\leq\overline{c}\rho_0
\end{equation}
with two positive constants $0<\underline{c}\leq\overline{c}<+\infty$ and obtained the global weak solutions to \eqref{drift-flux} in three dimensions. Under the assumption \eqref{ulbound}, the initial data may vanish but the vacuum of two phase must appear at the same point. 
For the case that one phase may persist as the other one vanishes, Evje-Wen-Zhu \cite{MR3606958} proved the existence of global weak solutions in the energy space
\begin{equation*}
	\left\{
	\begin{aligned}
	&0\leq \rho, n\in L^\infty\big(0,T;L^\infty(\Omega)\big),\\
	&w\in L^2\big(0,T;H^1_0(\Omega)\big),~~\sqrt{\rho+n}w\in L^\infty\big(0,T;L^2(\Omega)\big)
	\end{aligned}
	\right.
\end{equation*}
for any $T\in(0,\infty]$, where $\Omega=(0,1)$. For multi-dimensional case, the global existence of finite energy weak solutions to \eqref{drift-flux} was also obtained in \cite{MR3944205,MR3925535,MR4535977} with different pressure functions $P(\rho,n)$. 
The interested reader can refer to the review papers \cite{MR3830751,MR4614648} for the vicous two-fluid equations, while for the fluid-particles models, one can refer to \cite{MR2765745,MR4284416,MR4213662,MR4367916,MR2415460} and the references therein.
	
Nevertheless, there are few results about the well-posedness for the (pressureless) Euler-NS system (i.e., system \eqref{main0} with (without) $\nabla \rho$ in $\eqref{main0}_2$). The global existence of unique strong solutions to the Euler-NS system was established by Choi \cite{MR3546341} in the periodic domain and the whole space, under the assumption that the initial data is a small perturbation of the equilibrium state, in the energy space
\begin{equation*}
	\left\{
	\begin{aligned}
	&\rho-\bar{\rho}, n-\bar{n}\in C\big(0,T;H^s(\Omega)\big)\cap C^1\big(0,T;H^{s-1}(\Omega)\big), \\
	&u\in C\big(0,T;H^s(\Omega)\big)\cap C^1\big(0,T;H^{s-1}(\Omega)\big),\\
	&w\in C\big(0,T;H^s(\Omega)\big)\cap L^2\big(0,T;H^{s+1}(\Omega)\big),\ s>5/2
	\end{aligned}
	\right.
\end{equation*}
with constants $\bar{\rho}, \bar{n}>0$ for any $T\in(0,\infty]$, where $\Omega=\mathbb{T}^3\ \text{or}\ \mathbb{R}^3$. See also similar results \cite{MR4344262,MR4448314} in a bounded domain. Such results can be seen as an extension of the works of Matsumura-Nishida \cite{MR0555060,MR0564670,MR0713680} from the CNS system to the two-phase flows. 
The exponential stability with time in the periodic case was also proved in \cite{MR3546341}. Further, with some additional conditions on the initial data, Tang-Zhang \cite{MR4262061}, Wu-Zhang-Zou \cite{MR4175837}, and Zhang et al. \cite{MR4313091} obtained the optimal algebraic decay rates in time of the global solutions in the whole space with Sobolev regularity. Li-Shou \cite{MR4596744} investigated the local and global well-posedness of strong solutions to the Euler-NS system in a critical homogeneous Besov space and proved the optimal time decay rate, which reflects the influence of the relaxation drag force and the viscosity dissipation on the regularity of the solution. For the pressureless Euler-NS system \eqref{main0}, Choi-Kwon \cite{MR3487272} proved the local and global existence of classical solutions and established the exponential time stability in the three-dimensional spatial periodic space. The finite-time blow-up phenomena of classical solutions in multi-dimensional space were studied by Choi \cite{MR3723164}, and the result shows that any classical solutions in finite energy space cannot exist globally in time, provided some prescribed assumptions on the initial data and viscosity coefficients. 
	
In the present paper, we are going to investigate whether classical solutions of the Cauchy problem for the pressureless Euler-NS system exist in the inhomogeneous Sobolev space. The analysis of the existence of solutions to a two-phase flow is complicated by the fact that, due to the interaction between the two fluids, the system has quite different structure and properties for different relations between the densities of the two fluids. For some typical cases, we apply the maximum principle for a degenerate parabolic operator as well as for an elliptic type operator from the momentum equation in different cases depending on the sizes of the compact supports of the densities, and we prove that, with suitable assumptions on the initial data near the vacuum, the classical solution does not exist for any short time. 
	
The rest parts of this paper are organized as follows. In Section \ref{section2}, we present our main non-existence results Theorem \ref{thm1} and \ref{thm2} on the classical solutions to the Cauchy problem \eqref{main0}--\eqref{initial}. In Section \ref{section3}, we consider the problem in the Lagrangian coordinates and rewrite it into an initial boundary value problem. The proofs of Theorem \ref{thm1} and \ref{thm2} are given in Section \ref{section4} and Section \ref{section5} respectively.

\section{main results}\label{section2}

\subsection{Results}
Throughout this paper, we only consider the one-dimensional case, i.e., $d=1$. Let $C^1(0,T;H^m(\mathbb{R}))$ and $C_1^2(\mathbb{R}\times [0,T])$ be the standard Sobolev space and Schauder space (first (second) order continuously differentiable on time (space)), respectively. We assume that the initial densities $\rho_0$ and $n_0$ are compactly supported on open bounded sets $\Omega_1\subseteq\mathbb{R}$ and $\Omega_2\subseteq\mathbb{R}$ respectively. To be more specific, it holds
\begin{equation}\label{support}
	\begin{aligned}
	&{\rm{supp}}_x\rho_0=\bar{\Omega}_1,&&\rho_0(x)>0,&&x\in\Omega_1,\\
	&{\rm{supp}}_xn_0=\bar{\Omega}_2,&&n_0(x)>0,&&x\in\Omega_2.
	\end{aligned}
\end{equation}
For convenience, we further take
$$\Omega:=\Omega_1\cup\Omega_2=(0,1),$$
and denote
$$\mu(n):=2\nu(n)+\lambda(n).$$
Besides, the viscosity coefficient $\mu(n)$ is supposed to be $C^1$ such that
\begin{equation}\label{viscoe}
	\mu(n)\in [B^{-1},B]\quad\mathrm{and}\quad\mu'(n)\leq B
\end{equation}
with a positive constant $B>0$, and we assume
\begin{equation}\label{inda0}
	\begin{aligned}
	&\rho_0+n_0+|(n_0)_x|<m&&\text{in}~\Omega\\ 
	&|(u_0,(u_0)_x,(u_0)_{xx}, w_0,(w_0)_x,(w_0)_{xx})|< M&&\text{in}~\Omega
	\end{aligned}
\end{equation}
for some constants $m>0$ and $M>0$.

The main results can be stated as follows:
\begin{theorem}\label{thm1} 
Assume that the initial data $(\rho_0,u_0,n_0,w_0)$ satisfies \eqref{support} and \eqref{inda0}. When $\Omega_1\subsetneqq\Omega_2$, if there exist some constants $p_0>0$, $d_0\in(0,1)$ such that
\begin{equation}\label{inda1}
	\begin{aligned}
	\frac{(n_0)_x}{n_0}\geq p_0,~~w_0(d_0)<0,~~w_0\leq 0&&\mathrm{in}~(0,d_0]
	\end{aligned}
\end{equation}
or
\begin{equation}\label{inda2}
	\begin{aligned}
	\frac{(n_0)_x}{n_0}\leq -p_0,~~w_0(1-d_0)>0,~~w_0\geq 0&&\mathrm{in}~[1-d_0,1).
	\end{aligned}
\end{equation}
When $\Omega_1=\Omega_2$, except \eqref{inda1}--\eqref{inda2}, if the initial data additionally satisfies
\begin{equation}\label{inda3} 
	\begin{aligned}
	\frac{n_0^\gamma}{\rho_0}\geq\frac{2^{\gamma+4}M}{\gamma p_0}&&\mathrm{in}~(0,d_0)\cup (1-d_0,1).
	\end{aligned}
\end{equation}
Then \eqref{main0}--\eqref{initial} does not admit any solution $(\rho,u,n,w)\in C^1(0,T;H^m(\mathbb{R}))$ with $m>2$ for any positive time $T$.
\end{theorem}

\begin{theorem}\label{thm2} 
Assume that $\Omega_2\subsetneqq\Omega_1$ and the initial data $(\rho_0,u_0,n_0,w_0)$ satisfies \eqref{support} and \eqref{inda0}. If  there exist a constant $d_0\in(0,1)$ such that
\begin{equation}\label{inda5}
	\begin{aligned}
	w_0(d_0)<0,~~~~u_0-w_0<0\quad &&\mathrm{in}~(0,d_0],
	\end{aligned}
\end{equation}
or
\begin{equation}\label{inda6}
	\begin{aligned}
	w_0(1-d_0)>0,~~~~u_0-w_0>0\quad &&\mathrm{in}~[1-d_0,1).
	\end{aligned}
\end{equation}
Then \eqref{main0}--\eqref{initial} does not admit any solution $(\rho,u,n,w)\in C^1(0,T;H^m(\mathbb{R}))$ with $m>2$ for any positive time $T$.
\end{theorem}
	
\subsection{Remarks}	

The following remarks are helpful for understanding Theorem \ref{thm1} and \ref{thm2}.  
	
\begin{remark} 
(I)	Set
\begin{equation}\label{ex1}
	n_0(x)=
	\left\{ 
	\begin{aligned}
	&x^k(1-x)^k&& \ \mathrm{for} \ x\in [0,1],\\
	&0&&\ \mathrm{for}\ x\in \mathbb{R}\setminus [0,1],
	\end{aligned}
	\right.
\end{equation}
and
\begin{equation}\label{ex2}
	w_0(x)=
	\left\{ 
	\begin{aligned}
	&-x^l&& \ \mathrm{for}\ x\in [0,1/4],\\
	&\mathrm{smooth\ connection}&& \ \mathrm{for}\ x\in (1/4,3/4),\\
	&(1-x)^l&& \ \mathrm{for}\ x\in  [3/4,1],\\
	&0&& \ \mathrm{for}\ x\in \mathbb{R}\setminus [0,1].
	\end{aligned}
	\right.
\end{equation}
Then $(n_0,w_0)$ in \eqref{ex1} and \eqref{ex2} satisfies both \eqref{inda1} and \eqref{inda2} for any  $k,l\in \mathbb{N}_+$.
		
	(II)	Let 
		\begin{equation}\label{ex3}
			\rho_0(x)=
			\left\{ 
			\begin{aligned}
				&x^{\tilde{k}}(1-x)^{\tilde{k}}&& \ \mathrm{for} \ x\in [0,1],\\
				&0&&\ \mathrm{for}\ x\in \mathbb{R}\setminus [0,1].
			\end{aligned}
			\right.
		\end{equation}
		If we choose $\tilde{k}\in\mathbb{N}_+$ such that $k\gamma<\tilde{k}$, then $(\rho_0, n_0)$ in \eqref{ex1} and \eqref{ex3} satisfies \eqref{inda3}.

	(III)	 Let $u_0=\theta w_0$ with $\theta>1$. Then $(u_0,w_0)$ in \eqref{ex2} satisfies both \eqref{inda5} and \eqref{inda6}.

	\end{remark}

	\begin{remark}
		
		We mention that one can obtain similar non-existence results as in Theorem \ref{thm1} or \ref{thm2} for the Euler-NS system $($adding $\nabla \rho$ on the LHS of $\eqref{main0}_2)$ and the NS-NS system, since the proof entirely relies on the operator of the Navier-Stokes equations after Lagrange transformation.
		
	\end{remark}

     \begin{remark}
		
	The argument in this paper can also be applied to the spherically symmetric  pressureless Euler-NS system in higher dimension since the proof mainly relies on the maximum principle which is carried out on the  (single) momentum equation from the Navier-Stokes equations.   
		
	\end{remark}
	
	\section{Lagrangian formulation}\label{section3}
	Let $\eta_1$ and $\eta_2$ be the position of the fluid particle $x$ at time $t$ defined by
	\begin{equation}\label{LF-1}
		\begin{aligned}
			(\eta_1)_t(x,t)=u(\eta_1(x,t),t),\quad
			\eta_1(x,0)=x,
		\end{aligned}
	\end{equation}
	and
	\begin{equation}\label{LF-2}
		\begin{aligned}
			(\eta_2)_t(x,t)=w(\eta_2(x,t),t),\quad
			\eta_2(x,0)=x.
		\end{aligned}
	\end{equation}
	We denote by $\varrho_1, \varrho_2$ and $v_1, v_2$ the Lagrangian densities and velocities respectively, which satisfy
	\begin{equation}\label{LF-3}
		\begin{aligned}
			&\varrho_1(x,t)=\rho(\eta_1(x,t),t),&&v_1(x,t)=u(\eta_1(x,t),t),\\
			&\varrho_2(x,t)=n(\eta_2(x,t),t),&&v_2(x,t)=w(\eta_2(x,t),t).
		\end{aligned}
	\end{equation}
	In view of \eqref{LF-1}--\eqref{LF-3}, the Cauchy problem \eqref{main0}--\eqref{initial} is transformed into the following problem on $\Omega$: 
	\begin{equation}\label{main-LF-1}
		\left\{
		\begin{aligned}
			&(\varrho_1)_t+\varrho_1\frac{(v_1)_x}{(\eta_1)_x}=0 && \ \text{in} \ \Omega\times(0,T],\\
			&\varrho_1(v_1)_t=-\varrho_1(v_1-v_2) && \ \text{in} \ \Omega\times(0,T],\\
			&(\varrho_2)_t+\varrho_2\frac{(v_2)_x}{(\eta_2)_x}=0 && \ \text{in} \ \Omega\times(0,T],\\
			&\varrho_2(v_2)_t+\frac{(\varrho_2^\gamma)_x}{(\eta_2)_x}=\frac{1}{(\eta_2)_x}\left[\mu(\varrho_2)\frac{(v_2)_x}{(\eta_2)_x}\right]_x\\
			&\qquad\qquad\qquad\qquad+\varrho_1(v_1-v_2)&& \ \text{in} \ \Omega\times(0,T],\\
			&(\varrho_1, v_1,\varrho_2, v_2)(x,0)=(\rho_0, u_0, n_0, w_0)(x).
		\end{aligned}
		\right.
	\end{equation}
	
	Notice that $\varrho_1$ and $\varrho_2$ can be solved by $\eqref{main-LF-1}_1$ and $\eqref{main-LF-1}_3$ respectively as
	\begin{equation}\label{LF-4}
		\varrho_1=\frac{\rho_0}{(\eta_1)_x}\quad \mathrm{and}\quad \varrho_2=\frac{n_0}{(\eta_2)_x}. 
	\end{equation}
	Inserting \eqref{LF-4} into \eqref{main-LF-1}, we have
	\begin{equation}\label{main-LF-2} 
		\left\{
		\begin{aligned}
			&\rho_0(v_1)_t=-\rho_0(v_1-v_2)&& \ \text{in} \ \Omega\times(0,T],\\
			&n_0(v_2)_t+\left(\frac{n_0^\gamma}{(\eta_2)_x^\gamma}\right)_x=\left[\mu\left(\frac{n_0}{(\eta_2)_x}\right)\frac{(v_2)_x}{(\eta_2)_x}\right]_x\\
			&\qquad\qquad\qquad\qquad\qquad+\rho_0\frac{(\eta_2)_x}{(\eta_1)_x}(v_1-v_2)&& \ \text{in} \ \Omega\times(0,T],\\
			&(v_1, v_2)(x,0)=(u_0,w_0)(x).
		\end{aligned}
		\right.
	\end{equation}

	Let $T\leq \min\{1, \frac{1}{2M}, \frac{p_0}{4M}\}$ be a sufficiently small number\footnote{ $T$ can be taken sufficiently small since we will show Theorems \ref{thm1}--\ref{thm2} by contradiction.}. Assume that $(v_1,v_2)\in C^1([0,T];H^m(\mathbb{R}))$ with $m>2$ is the solution to \eqref{main-LF-2} with the initial data  satisfying \eqref{inda0}.
	Then it holds that
	\begin{equation}\label{LF-6}
		|(v_1,(v_1)_x,(v_1)_{xx},v_2,(v_2)_x,(v_2)_{xx})|\leq M\quad\text{in}~\Omega\times(0,T],
	\end{equation}
	and
	\begin{equation}\label{LF-7}
		\begin{aligned}
			&1/2\leq (\eta_1)_x,\ (\eta_2)_x\leq 2\quad &&\text{in}~\Omega\times(0,T],\\
			&|(\eta_1)_{xx}|,\ |(\eta_2)_{xx}|\leq \min\{1/2,p_0/4\}\quad &&\text{in}~\Omega\times(0,T].
		\end{aligned}
	\end{equation}
	Indeed, \eqref{LF-6} follows from the continuity of the solution $(v_1,v_2)$, and \eqref{LF-7} follows from \eqref{LF-1}--\eqref{LF-3} that
	\begin{equation*}
		\begin{aligned}
			&|(\eta_2)_x-1|\leq \int_0^t |(v_2)_x| ds\leq MT\leq 1/2,\\
			&|(\eta_2)_{xx}|\leq \int_0^t |(v_2)_{xx}| ds\leq MT\leq \min\{1/2,p_0/4\},
		\end{aligned}
	\end{equation*}
	and  other similar estimates.

	Due to \eqref{support}, for $x\in\mathbb{R}\setminus\bar{\Omega}$, one sees that $\rho_0(x)=n_0(x)=0$. Then it is derived from  $\eqref{main-LF-2}_2$ that
	\begin{equation*}
		\left[\mu\left(\frac{n_0}{(\eta_2)_x}\right)\frac{(v_2)_x}{(\eta_2)_x}\right]_x(x,t)=0\quad \mathrm{for}\ (x,t)\in\mathbb{R}\setminus\bar{\Omega}\times(0,T],
	\end{equation*}
	which together with \eqref{viscoe} and \eqref{LF-7} yields
	\begin{equation}\label{LF-8}
		(v_2)_x(x,t)=0\quad \mathrm{for}\ (x,t)\in\mathbb{R}\setminus\bar{\Omega}\times(0,T].
	\end{equation}
	Recalling $v_2\in C^1([0,T];H^m(\mathbb{R}))$ with $m>2$ and using \eqref{LF-8}, one obtains
	\begin{equation}\label{LF-9}
		v_2(x,t)=0\quad \mathrm{for}\ (x,t)\in\mathbb{R}\setminus\bar{\Omega}\times(0,T].
	\end{equation}
	
	With the new constraints \eqref{LF-8}--\eqref{LF-9}, to show that the Cauchy problem \eqref{main0}--\eqref{initial} has no solution in $ C^1([0,T];H^m(\mathbb{R}))$ with $m>2$, it suffices to show that the following initial boundary value problem:
	\begin{equation}\label{main-LF-3}
		\left\{
		\begin{aligned}
			&\rho_0(v_1)_t=-\rho_0(v_1-v_2)&& \ \text{in} \ \Omega\times(0,T],\\
			&n_0(v_2)_t+\left(\frac{n_0^\gamma}{(\eta_2)_x^\gamma}\right)_x=\left[\mu\left(\frac{n_0}{(\eta_2)_x}\right)\frac{(v_2)_x}{(\eta_2)_x}\right]_x\\
			&\qquad\qquad\qquad\qquad\qquad+\rho_0\frac{(\eta_2)_x}{(\eta_1)_x}(v_1-v_2)&& \ \text{in} \ \Omega\times(0,T],\\
			&v_2=(v_2)_x=0&&\ \text{on}~\partial\Omega\times(0,T],\\
			&(v_1, v_2)(x,0)=(u_0,w_0)(x)
		\end{aligned}
		\right.
	\end{equation}
	has no solution in $C_1^2(\Omega\times[0,T])$\footnote{This can be checked by contradiction using Sobolev embeddings.}.

	\section{Proof of Theorem \ref{thm1}}\label{section4}
	Define the linear degenerate parabolic operator
	\begin{equation}\label{Proof-1}
		\mathcal{L}v:=n_0v_t-a(x,t)v_{xx}-b(x,t)v_x\quad \text{in} \ \Omega\times(0,T]
	\end{equation}
	with
	\begin{equation}\label{Proof-TH1-2}
		a(x,t)=\mu\left(\frac{n_0}{(\eta_2)_x}\right)\frac{1}{(\eta_2)_x}\quad\mathrm{and}\quad b(x,t)=\left[\mu\left(\frac{n_0}{(\eta_2)_x}\right)\frac{1}{(\eta_2)_x}\right]_x.
	\end{equation}

	There are two cases to consider. 
	
	\underline{\emph{Case 1: $\Omega_1\subsetneqq\Omega_2$}}. In this case, 
	there exists a small number $d_0\in (0,1)$  such that $\rho_0=0$ but $n_0\neq 0$ for $x\in (0,d_0)\cup(1-d_0,1)$ and thus, by $\eqref{main-LF-3}_2$, it yields
	\begin{equation*}
		\begin{aligned}
			\mathcal{L}v_2=-\left(\frac{n_0^\gamma}{(\eta_2)_x^\gamma}\right)_x\quad \mathrm{for}\ (x,t)\in (0,d_0)\cup(1-d_0,1)\times(0,T].
		\end{aligned}
	\end{equation*}
	We regard the RHS (pressure) of the above equation as a forcing term. The key observation is that this forcing term has a fixed sign for $(x,t)\in (0,d_0)\times(0,T]$ that
	\begin{equation}\label{Proof-TH1-3}
		\begin{aligned}
			-\left(\frac{n_0^\gamma}{(\eta_2)_x^\gamma}\right)_x
			&=-\frac{\gamma n_0^\gamma}{(\eta_2)_x^\gamma}\left(\frac{(n_0)_x}{n_0}-\frac{(\eta_2)_{xx}}{(\eta_2)_x}\right)\\
			&\leq-\frac{\gamma n_0^\gamma}{2^{\gamma}}(p_0-p_0/2)\leq 0,
		\end{aligned}
	\end{equation}
	where one has used \eqref{inda1} and \eqref{LF-7}.
	Then it follows that
	\begin{equation}\label{Proof-TH1-5}
		\mathcal{L}v_2\leq 0\quad\text{in}~(0,d_0)\times(0,T].
	\end{equation}
	Similarly, one can show that
	\begin{equation}\label{Proof-TH1-6}
		\mathcal{L}v_2\geq 0\quad\text{in}~(1-d_0,1)\times(0,T]. 
	\end{equation}

	\underline{\emph{Case 2: $\Omega_1=\Omega_2$}}. In this case, 
	there exists a small number $d_0\in (0,1)$ such that
	\begin{equation}\label{Proof-TH1-6-add}
		\begin{aligned}
			&\mathcal{L}v_2=-\left(\frac{n_0^\gamma}{(\eta_2)_x^\gamma}\right)_x+\rho_0\frac{(\eta_2)_x}{(\eta_1)_x}(v_1-v_2)\\
			&\qquad\quad\mathrm{for}\ (x,t)\in (0,d_0)\cup(1-d_0,1)\times(0,T].
		\end{aligned}
	\end{equation}
	For the two forcing terms on the RHS of the above equation, we regard the pressure as a decisive term
	and obtain for $(x,t)\in (0,d_0)\times(0,T]$ that
	\begin{equation*}
		\begin{aligned}
			&-\left(\frac{n_0^\gamma}{(\eta_2)_x^\gamma}\right)_x+\rho_0\frac{(\eta_2)_x}{(\eta_1)_x}(v_1-v_2)\\
			=&-\rho_0\left[\frac{\gamma n_0^\gamma}{\rho_0(\eta_2)_x^\gamma}\left(\frac{(n_0)_x}{n_0}-\frac{(\eta_2)_{xx}}{(\eta_2)_x}\right)-\frac{(\eta_2)_x}{(\eta_1)_x}(v_1-v_2)\right]\\
			\leq&-\rho_0\left[\frac{\gamma n_0^\gamma}{2^{\gamma}\rho_0}(p_0-p_0/2)-8M\right]\leq 0
		\end{aligned}
	\end{equation*}
	where we have applied \eqref{inda3} and \eqref{LF-7}. Hence, \eqref{Proof-TH1-5} still holds, and one can handle similarly to have \eqref{Proof-TH1-6} as well.

	In the following, we mainly discuss the Hopf's lemma and the strong maximum principle for the operator $\mathcal{L}$ satisfying \eqref{Proof-TH1-5}. The other case \eqref{Proof-TH1-6} can be handled similarly. For notational simplicity, let $D_T$ be an open and bounded set in $(0,d_0)\times (0,T]$, and denote by $\partial_pD_T$ the parabolic boundary of $D_T$.

	\subsection{The Hopf's lemma}
	We start with the weak maximum principle.
	\begin{lemma} \label{dpweak}
		Assume that $v\in C_1^2(D_T)\cap C(\bar{D}_T)$.
		If $v$ satisfies
		\begin{equation}\label{Proof-TH1-7}
			\mathcal{L}v\leq 0\quad\mathrm{in}~D_T,
		\end{equation}
		and $v$ attains its maximum at $(x_0,t_0)$, then $(x_0,t_0)\in\partial_p D_T$. 
	\end{lemma}
	\begin{proof} The proof is close to \cite{MR3925527}, so we omit it.
	\end{proof}
	
	Next, we state the Hopf's lemma.
	\begin{lemma}\label{dpHopf} Assume that $v\in C_1^2\big((0,d_0)\times (0,T]\big)\cap C\big([0,d_0]\times [0,T]\big)$.
		If $v$ satisfies \eqref{Proof-TH1-7}, and $v(x,t)<v(0,t_0)$ for any $(x,t)$ in the neighborhood $D_T$ of $(0,t_0)$, 
		defined by
		\begin{equation*}
			\begin{aligned}
				D_T=\{(x,t):(x-\ell)^2+(t_0-t)<\ell^2,~0<x<\ell/2,~0<t\leq t_0\},
			\end{aligned}
		\end{equation*}
		where $\ell\in (0,d_0)$, and $t_0-3\ell^2/4>0$. Then it holds that
		\begin{equation}\label{Proof-TH1-8}
			\frac{\partial v(0,t_0)}{\partial \vec{n}}>0.
		\end{equation}
		where $\vec{n}$ is the outer unit normal vector at the point $(0, t_0)$.
	\end{lemma}
	
	\begin{proof} 
		For $(x,t)\in D_T$, define the auxiliary functions
		\begin{equation*}
			q(\alpha,x,t)=e^{-\alpha\left[(x-\ell)^2+(t_0-t)\right]}-e^{-\alpha\ell^2},
		\end{equation*}
		and
		\begin{equation*}
			\varphi(\varepsilon,\alpha,x,t)=v(x,t)-v(0,t_0)+\varepsilon q(\alpha,x,t),
		\end{equation*}
		where $\alpha$ and $\varepsilon$ are constants to be determined later. 
		
		\underline{\emph{Step 1: Fixing $\varepsilon$}.} Note that 
		$\partial_p D_T$ consists of 
		\begin{equation*}
			\Sigma_1:=\{(x,t):(x-\ell)^2+(t_0-t)<\ell^2,~x=\ell/2,~0<t\leq t_0\}, 
		\end{equation*}
		and
		\begin{equation*}
			\Sigma_2:=\{(x,t):(x-\ell)^2+(t_0-t)=\ell^2,~0\leq x\leq \ell/2,~0<t\leq t_0\}.
		\end{equation*}

		First, it is easy to see that
		\begin{equation*}
			\begin{aligned}
				&\varphi(\varepsilon,\alpha,0,t_0)=0,\\
				&0\leq q(\alpha,x,t)\leq 1\quad &&\mathrm{on}\ \Sigma_1,\\
				&q(\alpha,x,t)=0\quad \mathrm{and}\quad v(x,t)-v(0,t_0)\leq 0\quad &&\mathrm{on}\ \Sigma_2. 
			\end{aligned}
		\end{equation*}
		Next, one can choose $\varepsilon_0\in(0,1)$ sufficiently small such that 
		$$v(x,t)-v(0,t_0)\leq-\varepsilon_0<0\quad \mathrm{on}\ \Sigma_1.$$ 
		Therefore, for such an $\varepsilon_0$, it holds that
		\begin{equation*}
			\left\{ 
			\begin{aligned}
				&\varphi(\varepsilon_0,\alpha,x,t)\leq0 \quad \text{on} \ \partial_pD_T,\\
				&\varphi(\varepsilon_0,\alpha,0,t_0)=0.
			\end{aligned}
			\right.
		\end{equation*}

		\underline{\emph{Step 2: Fixing $\alpha$}.} It follows from basic calculations that
		\begin{equation}\label{Proof-TH1-9}
			\begin{aligned}
				\mathcal{L}q&=n_0q_t-a(x,t)q_{xx}-b(x,t)q_x\\
				&= \left\{-4a(x-\ell)^2\alpha^2+[n_0+2a+2b(x-\ell)]\alpha\right\}e^{-\alpha\left[(x-\ell)^2+(t_0-t)\right]}.
			\end{aligned}
		\end{equation}
		Then, one may estimate
		\begin{equation}\label{Proof-TH1-10}
			\begin{aligned}
				n_0\leq m,\quad
				-\ell<x-\ell\leq -\ell/2,
			\end{aligned}
		\end{equation}
		\begin{equation}\label{Proof-TH1-11}
			\begin{aligned}
				1/(2B)\leq a(x,t)=\mu\left(\frac{n_0}{(\eta_2)_x}\right)\frac{1}{(\eta_2)_x}\leq 2B,
			\end{aligned}
		\end{equation}
		and
		\begin{equation}\label{Proof-TH1-12}
			\begin{aligned}
				|b(x,t)|&=\mu^\prime\left(\frac{n_0}{(\eta_2)_x}\right)\frac{1}{(\eta_2)_x}\left[\frac{(n_0)_x}{(\eta_2)_x}-\frac{n_0(\eta_2)_{xx}}{(\eta_2)_x^2}\right]\\
				&\quad+\mu\left(\frac{n_0}{(\eta_2)_x}\right)\left(-\frac{(\eta_2)_{xx}}{(\eta_2)_x^2}\right)\\
				&\leq B\cdot 2\cdot (2m+4mMT)+B\cdot (4MT)\leq 8B(m+1),
			\end{aligned}
		\end{equation}
		where one has used $T\leq\frac{1}{2M}$ and \eqref{LF-7}.
		Hence, the combination of \eqref{Proof-TH1-9}--\eqref{Proof-TH1-12} leads to
		\begin{equation*}
			\begin{aligned}
				\mathcal{L}q
				\leq \bigg[-\frac{\ell^2}{2B}\alpha^2+(m+4B+8B(m+1)\ell)\alpha\bigg]e^{-\alpha\left[(x-\ell)^2+(t_0-t)\right]},
			\end{aligned}
		\end{equation*}
		which implies that there exists $\alpha_0=\alpha_0(\ell,B,m)>0$ large enough such that
		\begin{equation*}
			\mathcal{L}q(\varepsilon_0,\alpha_0,x,t)\leq 0\quad \mathrm{in}~D_T.
		\end{equation*}
		This together with \eqref{Proof-TH1-7} gives
		\begin{equation*}
			\begin{aligned}
				\mathcal{L}\varphi(\varepsilon_0,\alpha_0,x,t)\leq \mathcal{L}v(x,t)+\varepsilon_0\mathcal{L}q(\alpha_0,x,t)
				\leq \varepsilon_0\mathcal{L}q(\alpha_0,x,t)\leq 0\quad \mathrm{in}~D_T.
			\end{aligned}
		\end{equation*}
		
		Therefore, we conclude that
		\begin{equation}\label{Proof-TH1-13}
			\left\{ 
			\begin{aligned}
				&\mathcal{L}\varphi(\varepsilon_0,\alpha_0,x,t)\leq 0\quad &&\text{in}~D_T,\\
				&\varphi(\varepsilon_0,\alpha_0,x,t)\leq 0\quad &&\text{on}~\partial_p D_T,\\
				&\varphi(\varepsilon_0,\alpha_0,0,t_0)=0,
			\end{aligned}
			\right.
		\end{equation}
		which, together with Lemma \ref{dpweak}, deduces that $\varphi(\varepsilon_0,\alpha_0,x,t)$ attains its maximum at $(0,t_0)$. In particular,
		\begin{align*}
			\varphi(\varepsilon_0,\alpha_0,x,t_0)\leq \varphi(\varepsilon_0,\alpha_0,0,t_0)\quad  \text{for}\  x\in (0, \ell/2),
		\end{align*}
		and this implies
		\begin{align*}
			\frac{\partial \varphi(\varepsilon_0,\alpha_0,0,t_0)}{\partial \vec{n}}\geq 0.
		\end{align*}
		It follows that
		\begin{align*}
			\frac{\partial v(0,t_0)}{\partial \vec{n}}\geq-\varepsilon_0 \frac{\partial q(\alpha_0,0,t_0)}{\partial \vec{n}}
			=2\varepsilon_0\alpha_0 \ell e^{-\alpha_0 \ell^2}>0,
		\end{align*}
		and we complete the proof.
	\end{proof}
	
	\subsection{The strong maximum principle}
	To establish the strong maximum principle, one first needs to study the t-derivative of the interior maximum point. 
	
	\begin{lemma}\label{dpt-derivative}
		Assume that $v\in C_1^2\big((0,d_0)\times (0,T]\big)\cap C\big([0,d_0]\times [0,T]\big)$.
		If $v$ satisfies \eqref{Proof-TH1-7}, and attains its maximum in an interior point $P_0=(x_0,t_0)$ of $(0,d_0)\times (0,T]$, then
		$v(P)=v(P_0)$ for any $P=(x,t_0)$ in $(0,d_0)\times (0,T]$. 
	\end{lemma}
	
	Lemma \ref{dpt-derivative} is a direct consequence of the following proposition by a standard argument \cite{MR0181836}. 
	\begin{proposition}\label{dpellipsoid}
		Suppose that  $v\in C_1^2\big((0,d_0)\times (0,T]\big)\cap C\big([0,d_0]\times [0,T]\big)$ satisfies \eqref{Proof-TH1-7} and has a maximum $M_0$ in $(0,d_0)\times (0,T]$. If  $(0,d_0)\times (0,T]$ contains a closed solid ellipsoid
		\begin{align*}
			D_T^\sigma:=\big\{(x,t): (x-x_*)^2+\sigma(t-t_*)^2\leq r^2,\  \sigma>0\big\}
		\end{align*}
		and  $v(x,t)<M_0$ for any interior point $(x,t)$ of $D_T^\sigma$ and  $v(\bar{x},\bar{t})=M_0$ at some $(\bar{x},\bar{t})\in \partial_p D_T^\sigma$. Then $\bar{x}=x_*$.
	\end{proposition}
	
	\begin{proof}[Proof of Proposition \ref{dpellipsoid}]  Similar to the reduction in \cite{MR3925527},  one may assume that $v$ attains the maximum $M_0$ in $D_T^\sigma$ at no more than two isolated points $(\bar{x},\bar{t})$ and $(\tilde{x},\tilde{t})$ on the boundary $\partial_p D_T^\sigma$.  
		
		We will prove the desired result by contradiction. 
		Suppose $\bar{x}\neq x_*$. First, it is easy to see that $\bar{t}<T$ by applying Lemma \ref{dpweak} on  $[0,d_0]\times [0,T]$. Next, choosing 
		a  closed ball $B_T((\bar{x},\bar{t}), \tilde{r})\subseteq (0,d_0)\times (0,T]$ with center $(\bar{x},\bar{t})$ and radius $\tilde{r}<\min\{|\bar{x}-x_*|, |\bar{x}-\tilde{x}|\}$, then one has 
		\begin{equation}\label{Proof-TH1-14}
			\begin{aligned}
				|x-x_*|\geq|\bar{x}-x_*|-\tilde{r}=:\hat{r}>0
			\end{aligned}
		\end{equation}
		for any  $(x,t)\in B_T((\bar{x},\bar{t}), \tilde{r})$.

		Define the auxiliary functions 
		\begin{align*}
			q(\alpha,x,t)=e^{-\alpha[(x-x_*)^2+\sigma(t-t_*)^2]}-e^{-\alpha r^2},
		\end{align*}
		and
		\begin{align*}
			\varphi(\varepsilon,\alpha,x,t)=w(x,t)-M_0+\varepsilon q(\alpha,x,t)
		\end{align*}
		with $\alpha$ and $\varepsilon$ to be determined later.

		By performing some careful calculations similar to the proof of Lemma \ref{dpHopf} by using \eqref{Proof-TH1-14}, 
		one can choose $\varepsilon_0>0$ and $\alpha_0>0$ such that
		\begin{equation}\label{Proof-TH1-15}
			\left\{ 
			\begin{aligned}
				&\mathcal{L}\varphi(\varepsilon_0,\alpha_0,x,t)\leq 0&&\text{in}~B_T((\bar{x},\bar{t}), \tilde{r}),\\
				&\varphi(\varepsilon_0,\alpha_0,x,t)< 0&&\text{on}~\partial_p B_T((\bar{x},\bar{t}), \tilde{r}),\\
				&\varphi(\varepsilon_0,\alpha_0,\bar{x},\bar{t})=0.
			\end{aligned}
			\right.
		\end{equation}
		Applying Lemma  \ref{dpweak} to $\eqref{Proof-TH1-15}_1$--$\eqref{Proof-TH1-15}_2$ leads to
		\begin{align*}
			\varphi(\varepsilon_0,\alpha_0,x,t)<0\quad  \text{in}\ B_T((\bar{x},\bar{t}), \tilde{r})
		\end{align*}
		which contradicts $\eqref{Proof-TH1-15}_2$ due to $(\bar{x},\bar{t})\in B_T((\bar{x},\bar{t}), \tilde{r})$.
	\end{proof}
	
	Next, we present a localized strong maximum principle, which reads as:
	\begin{lemma}\label{dplsm}
		Suppose that $v\in C_1^2\big((0,d_0)\times (0,T]\big)\cap C\big([0,d_0]\times [0,T]\big)$ satisfies \eqref{Proof-TH1-7}.
		If $v$ attains its maximum in the interior point $P_0=(x_0,t_0)$ of $(0,d_0)\times (0,T]$, then there exists a rectangle
		\begin{align*}
			\mathcal{R}(P_0):=\big\{(x,t): x_0-a_1\leq x\leq x_0+a_1,\  t_0-a_0\leq t\leq t_0\big\}
		\end{align*}
		in $(0,d_0)\times (0,T]$ such that
		$v(P)=v(P_0)$ for any $P\in \mathcal{R}(P_0)$.
	\end{lemma}
	\begin{proof} We will prove the desired result by contradiction. To this end, we suppose that there exists an interior point $P_1=(x_1,t_1)\in (0,d_0)\times (0,T]$ with $t_1< t_0$ such that $v(P_1)<v(P_0)$. 
		
		Connecting  $P_1$ to $P_0$ by a simple smooth curve $\Gamma$, one notices that there exists  $P_*=(x_*,t_*)$ on $\Gamma$ such that $v(P_*)=v(P_0)$ and $v(\bar{P})<v(P_*)$ for any $\bar{P}=(\bar{x},\bar{t})\in\Gamma$ between $P_1$ and $P_*$. 
		
		Assuming that $P_*=P_0$ and $P_1$ is very near to $P_0$, then one may find that there exists a rectangle $\mathcal{R}(P_0)$ in $(0,d_0)\times (0,T]$ with small positive numbers $a_0$ and $a_1$ (will be determined) such that $P_1$ lies on $t=t_0-a_0$.
		
		Since $\mathcal{R}(P_0)\setminus \{t=t_0\}\cap\{t=\bar{t}\}$ contains some $\bar{P}\in \Gamma$ and $v(\bar{P})<v(P_0)$,  it follows from Lemma \ref{dpweak} that
		\begin{equation*}
			\begin{aligned}
				v(P)<v(P_0)\quad \mathrm{for\ any}\ P \in \mathcal{R}(P_0)\setminus \{t=t_0\}\cap\{t=\bar{t}\}.  
			\end{aligned}
		\end{equation*}
		Hence, it holds that $v(P)<v(P_0)$ for any $P \in \mathcal{R}(P_0)\setminus \{t=t_0\}$.
		
		Consider two auxiliary functions 
		\begin{align*}
			q(\alpha,x,t)=t_0-t-\alpha(x-x_0)^2
		\end{align*}
		and
		\begin{align*}
			\varphi(\varepsilon,\alpha,x,t)=v(x,t)-v(P_0)+\varepsilon q(\alpha,x,t)
		\end{align*}
		with parameters $\alpha$ and $\varepsilon$ to be determined later.
		
		Assume further that $P=(x_0-a_1,t_0-a_0)$ is on the parabola $q(\alpha,x,t)=0$, then it holds that
		\begin{align}\label{Proof-TH1-16}
			\alpha=a_0a_1^{-2}.
		\end{align}

		\underline{\emph{Step 1: Fixing $a_1$}.} We first choose
		\begin{align}\label{Proof-TH1-17}
			a_1<\min\{x_0, d_0-x_0\}.
		\end{align}
		
		\underline{\emph{Step 2: Fixing $\alpha$}.}	Note that
		\begin{equation*}
			\begin{aligned}
				\mathcal{L}q&=n_0q_t-a(x,t)q_{xx}-b(x,t)q_x\\
				&=-n_0+2[a+b(x-x_0)]\alpha\\
				&\leq -n_0+(4B+16B(m+1)a_1)\alpha.
			\end{aligned}
		\end{equation*}
		Since $n_0$ has a positive lower bound depending on $x_0-a_1$ in $\mathcal{R}(P_0)$,  one may choose 	
		\begin{align}\label{Proof-TH1-18}
			\alpha_0<\frac{n_0}{4B+16B(m+1)a_1}
		\end{align}
		such that 
		\begin{align}\label{Proof-TH1-19}
			\mathcal{L}\varphi(\varepsilon,\alpha_0,x,t)\leq \mathcal{L}v(x,t)+\mathcal{L}q(\alpha_0,x,t)
			\leq 0 \quad \text{in}\ \mathcal{R}(P_0).
		\end{align}

		\underline{\emph{Step 3: Fixing $a_0$}.} 
		Combing \eqref{Proof-TH1-16}--\eqref{Proof-TH1-18}, we can choose  
		\begin{align*}
			a_0<\min\bigg\{t_0,\ \frac{n_0 a_1^2}{4B+16B(m+1)a_1}\bigg\}.
		\end{align*}

		\underline{\emph{Step 4: Fixing $\varepsilon$}.} Set 
		$\mathcal{S}_T=\{(x,t)\in \mathcal{R}(P_0): q(\alpha_0,x,t)\geq0\}$.	
		Observe that $\partial_p\mathcal{S}_T$ consists of two parts: 
		\begin{equation*}
			\begin{aligned}
				\Sigma_1\ \mathrm{lying\ in}\ \mathcal{R}(P_0)\quad \mathrm{and}\quad \Sigma_2\  \mathrm{lying\ on}\  \mathcal{R}(P_0)\cap \{t=t_0-a_0\}.
			\end{aligned}
		\end{equation*}
		
		First, one finds that
		\begin{equation*}
			\begin{aligned}
				&\varphi(\varepsilon,\alpha_0,x_0,t_0)=0,\\
				&q(\alpha_0,x,t)\ \mathrm{is\ bounded}\ &&\mathrm{on}\ \Sigma_2,\\
				&q(\alpha_0,x,t)=0\quad \mathrm{and}\quad v(x,t)-v(P_0)< 0\quad &&\mathrm{on}\ \Sigma_1\setminus \{P_0\}. 
			\end{aligned}
		\end{equation*}
		Next, one can choose $\varepsilon_0\in(0,1)$ sufficiently small such that 
		\begin{equation*}
			\begin{aligned}
				v(x,t)-v(P_0)\leq-\varepsilon_0<0\quad \mathrm{on}\ \Sigma_2.
			\end{aligned}
		\end{equation*}
		Hence, it holds that
		\begin{equation}\label{Proof-TH1-20}
			\left\{ \begin{aligned}
				&\varphi(\varepsilon_0,\alpha_0,x,t)<0&& \mathrm{on}\  \partial_p\mathcal{S}_T\setminus \{P_0\},\\
				&\varphi(\varepsilon_0,\alpha_0,x_0,t_0)=0.
			\end{aligned}
			\right.
		\end{equation}

		Now,  it follows from \eqref{Proof-TH1-16} and \eqref{Proof-TH1-20} that there exist $\varepsilon_0$ and $\alpha_0$ such that
		\begin{equation}\label{Proof-TH1-21}
			\left\{ 
			\begin{aligned}
				&\mathcal{L}\varphi(\varepsilon_0,\alpha_0,x,t)\leq 0&&\quad \text{in}\ \mathcal{S}_T,\\
				&\varphi(\varepsilon_0,\alpha_0,x,t)<0&&\quad  \text{on}\  \partial_p\mathcal{S}_T\setminus \{P_0\},\\
				&\varphi(\varepsilon_0,\alpha_0,x_0,t_0)=0.
			\end{aligned}
			\right.
		\end{equation}
		In view of Lemma \ref{dpweak} and \eqref{Proof-TH1-21}, $\varphi(\varepsilon_0,\alpha_0,x,t)$  only attains its maximum at $(x_0,t_0)$ in $\mathcal{\bar{S}}_T$. Therefore, one has
		\begin{align*}
			\varphi_t(\varepsilon_0,\alpha_0,x_0,t_0)\geq 0.
		\end{align*}
		Note that $q$ satisfies
		\begin{align*}
			q_t(\alpha_0,x_0,t_0)=-1.
		\end{align*}
		The above two facts yield
		\begin{align}\label{Proof-TH1-22}
			v_t(x_0,t_0)=\varphi_t(\varepsilon_0,\alpha_0,x_0,t_0)-\varepsilon_0q_t(\alpha_0,x_0,t_0)\geq \varepsilon_0.
		\end{align}
		
		On the other hand, since $v$ attains its maximum at $(x_0,t_0)$, it follows that
		\begin{align*}
			n_0v_t(x_0,t_0)\leq a v_{xx}(x_0,t_0)+b v_x(x_0,t_0)\leq 0,
		\end{align*}
		which contradicts \eqref{Proof-TH1-22}. 
	\end{proof}

	Now, we have the strong maximum principle. 
	\begin{lemma}\label{dpsm}
		Suppose that $v\in C_1^2(D_T)\cap C(\bar{D}_T)$ satisfies \eqref{Proof-TH1-7}.
		Then $v$ attains its maximum on  $\partial_p D_T$.
	\end{lemma}
	\begin{proof}
		The proof can be routinely completed via Lemma \ref{dpt-derivative} and  Lemma \ref{dplsm}, and thus we omit here. One can refer to \cite{MR0181836, MR3925527} for more details.
	\end{proof}
	
	Finally, we finish the proof of Theorem \ref{thm1}.
	\begin{proof}[Proof of Theorem \ref{thm1}] First, due to $w_0(d_0)<0$ given in \eqref{inda5},  and the continuity on time of the solution,  there exists a time $t_0\in (0,T]$ such that 
		$$v_2(d_0,t)<0\quad \mathrm{for}\ t\in (0,t_0].$$
		For such domain $(0,d_0)\times (0,t_0]$, it follows from Lemma \ref{dpweak} that  $v_2$  attains its  maximum on the set
		$$[0,d_0]\times\{t=0\}\cup\{x=0\}\times (0,t_0]\cup \{x=d_0\}\times (0,t_0].$$
		
		Next, recalling \eqref{inda5} that $w_0\leq0$ in $[0,d_0]$, and $v_2=0$ on $\{x=0\}\times (0,t_0]$ due to $\eqref{main-LF-3}_3$, by Lemma \ref{dpsm}, one concludes that $v_2$  attains its  maximum on the set 
		$$[0,d_0]\times\{t=0\}\cup\{x=0\}\times (0,t_0].$$  
		
		Hence
		$$v_2(x,t)<v_2(0,t_0)\quad \text{for}\ (x,t)\in (0,d_0)\times (0,t_0].$$
		This together with Lemma \ref{dpHopf} yields
		$$\frac{\partial v_2(0,t_0)}{\partial \vec{n}}>0,$$ 
		which contradicts $(v_2)_x=0$ on $\partial \Omega \times (0,T]$ in $\eqref{main-LF-3}_3$. 
	\end{proof}

	\section{Proof of Theorem \ref{thm2}}\label{section5}
	
	Noticing that in the case $\Omega_2\subsetneqq\Omega_1$,  there exists a small number $d_0\in (0,1)$ such that $n_0=0$ but $\rho_0\neq 0$ for $x\in (0,d_0)\cup(1-d_0,1)$ and thus, by $\eqref{main-LF-3}_2$, we have
	\begin{equation}\label{Proof-TH2-1}
		\begin{aligned}
			L_tv_2=\rho_0\frac{(\eta_2)_x}{(\eta_1)_x}(v_1-v_2),
		\end{aligned}
	\end{equation}
	where the operator $L_t$ is defined by
	\begin{equation*}
		\begin{aligned}
			L_t=-a(x,t)v_{xx}+b(x,t)v_x
		\end{aligned}
	\end{equation*}
	with 
	\begin{equation*}
		\begin{aligned}
			a(x,t)=\mu(0)\frac{1}{(\eta_2)_x}\quad \text{and}\quad 
			b(x,t)=\mu(0)\frac{(\eta_2)_{xx}}{(\eta_2)_x^2}. 
		\end{aligned}
	\end{equation*}
	
	Again, we regard the RHS of \eqref{Proof-TH2-1} as a forcing term.  Different with \eqref{Proof-TH1-6-add}, the pressure term vanishes in \eqref{Proof-TH2-1}, so we instead use \eqref{inda5} and the continuity of the solution to deduce 
	\begin{equation}\label{Proof-TH2-2}
		\begin{aligned}
			\rho_0\frac{(\eta_2)_x}{(\eta_1)_x}(v_1-v_2)\leq 0
		\end{aligned}
	\end{equation}
	for $(x,t)\in (0,d_0)\times(0,T]$ since $T>0$ is assumed to be sufficiently small. 
	Hence, for any fixed $t\in (0,T]$, it follows from \eqref{Proof-TH2-1} and \eqref{Proof-TH2-2} that 
	\begin{equation}\label{Proof-TH2-3}
		L_tv_2\leq 0\quad\text{in}~(0,d_0).
	\end{equation}
	One can also show similarly, for any fixed $t\in (0,T]$, that
	\begin{equation}\label{Proof-TH2-4}
		L_tv_2\geq 0\quad\text{in}~(1-d_0,1). 
	\end{equation}

	In the following, we mainly focus on the Hopf's lemma and the strong maximum principle for the operator $L_t$ satisfying \eqref{Proof-TH2-3} since the other scenario \eqref{Proof-TH2-4} can be handled in a similar fashion. For ease of notation, we assume $D$ to be an open and bounded set in $(0,d_0)$, and denote by $\partial D$ the boundary of $D$. 
	
	To begin with, we show the weak maximum principle. 
	\begin{lemma}\label{elweak}
		Assume that $v\in C_1^2(D)\cap C(\bar{D})$. 
		For fixed $t\in (0,T]$, if $v$ satisfies 
		\begin{equation}\label{Proof-TH2-5}
			L_tv\leq 0\quad\mathrm{in}~D,
		\end{equation}
		and $v$ attains its maximum at $x_0$, then $x_0\in \partial D$.
	\end{lemma}
	\begin{proof}
		Let us first suppose that $L_tv<0$ in $D$, and there exists $x_0\in D$ such that 
		\begin{equation*}
			v(x_0)=\max_{\bar{D}}v.
		\end{equation*}
		Then at this maximum point $x_0$, one has
		\begin{equation*}
			v_x(x_0)=0,~~v_{xx}(x_0)\leq 0,
		\end{equation*}
		which implies $L_tv\geq 0$ at the point $x_0$, and this leads to a contradiction. 
		
		Generally, we consider the auxiliary function $v^\varepsilon:=v+\varepsilon e^{\beta x}$ with $\varepsilon>0$ and $\beta>0$. For fixed $t\in (0,T]$, it follows from $T\leq\frac{1}{2M}$ and \eqref{LF-7} that
		\begin{equation}\label{Proof-TH2-6}
			\begin{aligned}
				&1/(2B)\leq a(x,t)=\mu(0)\frac{1}{(\eta_2)_x}\leq 2B,\\
				&|b(x,t)|\leq \mu(0)\frac{|(\eta_2)_{xx}|}{(\eta_2)_x^2}
				\leq B\cdot (4MT)\leq 2B.
			\end{aligned}
		\end{equation}
		Therefore it holds from \eqref{Proof-TH2-5} and \eqref{Proof-TH2-6} that
		\begin{equation*}
			\begin{aligned}
				L_tv^\varepsilon&=L_tv+\varepsilon L_t(e^{\beta x})
				\leq \varepsilon\left(-a\beta^2+b\beta\right)e^{\beta x}\\
				&\leq \varepsilon\left(-\frac{1}{2B}\beta^2+2B\beta\right)e^{\beta x}
				\leq -\varepsilon e^{\beta x}<0\quad\text{in}~D
			\end{aligned}
		\end{equation*}
		for sufficiently large $\beta>0$. Thus $v^\varepsilon$ attains its maximum on $\partial D$. Passing the limit $\varepsilon\rightarrow 0$ proves the desired result.
	\end{proof}
	
	Next, we establish the Hopf's lemma.
	
	\begin{lemma}\label{elHopf}
		Assume that $v\in C^2(D)\cap C(\bar{D})$. For fixed $t\in (0,T]$, 
		if $v$ satisfies \eqref{Proof-TH2-5}, and there exists a neighbourhood $D$ of $x=0$ such that
		\begin{equation}\label{Proof-TH2-7}
			v(0, t)>v(x, t)\quad\mathrm{for\ all}~x\in D. 
		\end{equation}
		Then it holds that
		\begin{equation*}
			\frac{\partial v(0, t)}{\partial \vec{n}}>0.
		\end{equation*}
		where $\vec{n}$ is the outer unit normal vector at the point $x=0$.
	\end{lemma}
	\begin{proof}
		For any $x, \ell\in D$, define
		\begin{equation*}
			\varphi(\beta,x)=e^{-\beta(x-\ell)^2}-e^{-\beta \ell^2},
		\end{equation*}
		where $\beta>0$ will be selected below. By \eqref{Proof-TH2-5}, one may estimate
		\begin{equation*}
			\begin{aligned}
				L_t\varphi(\beta,x)&=\left(-4a(x-\ell)^2\beta^2+2a\beta+2b(x-\ell)\beta\right)e^{-\beta(x-\ell)^2}\\
				&\leq \left(-\frac{\ell^2}{2B}\beta^2+(4B+4B\ell)\beta\right)e^{-\beta(x-\ell)^2}\\
				&\leq 0\quad \text{for}\ x\in(0,\ell/2),
			\end{aligned}
		\end{equation*}
		provided $\beta\geq \beta_0:=\beta_0(\ell,B)>0$.   For such a $\beta_0$,  it follows that
		\begin{equation}\label{Proof-TH2-8}
			L_t\big(v(x, t)-v(0, t)+\epsilon_0 \varphi(\beta_0,x)\big)\leq 0\quad\text{for}\ x\in (0,\ell/2). 
		\end{equation}
		
		In view of \eqref{Proof-TH2-7}, there exists a small constant 
		$\varepsilon_0>0$ such that
		\begin{equation*} 
			v(0, t)\geq v(x, t)+\varepsilon_0\varphi(\beta_0,x)\quad \text{for}\ x=\ell/2.
		\end{equation*}
		On the other hand, it holds that
		\begin{equation}\label{Proof-TH2-9}
			v(x, t)-v(0, t)+\varepsilon_0\varphi(\beta_0,x)= 0\quad \text{for}\ x=0. 
		\end{equation}
		
		Applying Lemma \ref{elweak} to \eqref{Proof-TH2-7}--\eqref{Proof-TH2-8}, we obtain
		\begin{equation*}
			v(x, t)-v(0, t)+\varepsilon_0 \varphi(\beta_0,x)\leq 0\quad\text{for}\ x\in (0, \ell/2).
		\end{equation*}
		This together with \eqref{Proof-TH2-9} yields
		\begin{equation*}
			\frac{\partial v(0, t)}{\partial \vec{n}}\geq -\varepsilon_0\frac{\partial \varphi(\beta_0,0)}{\partial \vec{n}}\geq 2\varepsilon_0\beta_0\ell e^{-\ell^2}>0,
		\end{equation*}
		and the proof is completed.
	\end{proof}
	
	The following strong maximum principle is direct consequence of the Hopf's lemma, whose proof can be found
	for instance in \cite{MR1625845}.
	\begin{lemma}\label{elstrong}
		Assume that $v\in C^2(D)\cap C(\bar{D})$. For fixed $t\in (0,T]$, 
		if $v$ satisfies \eqref{Proof-TH2-5}, and attains its maximum at an interior point $x_0\in D$, then $v(x, t)\equiv v(x_0,t)$ for all $x\in D$. 
	\end{lemma}
	
	Now, we are ready to show Theorem \ref{thm2}. 
	\begin{proof}[Proof of Theorem \ref{thm2}] First, for any fixed $t\in (0,T]$, it follows from Lemma \ref{elweak} that  $v_2(\cdot, t)$  attains its  maximum on the set
		$$\{x=0\}\cup \{x=d_0\}.$$
		On one hand, by $w_0(d_0)<0$ due to \eqref{inda5},  and the continuity on time of the solution,  there exists a time $t_0\in (0,T]$ such that 
		$$v_2(d_0,t)<0\quad \mathrm{for}\ t\in (0,t_0].$$
		On the other hand, thanks to $\eqref{main-LF-3}_3$, $v_2(0,t)=0$ for $ t\in (0,t_0]$.
		
		Hence, by Lemma \ref{elstrong}, $v_2(\cdot, t)$  attains its  maximum on the set 
		$\{x=0\}$, namely, for fixed $t\in(0,t_0]$, 
		$$v_2(x,t)<v_2(0,t)\quad \text{for}\ x\in (0,d_0),$$
		which, together with Lemma \ref{elHopf}, implies 
		$$\frac{\partial v_2(0,t)}{\partial \vec{n}}>0.$$ 
		This leads to a contradiction with $(v_2)_x=0$ on $\partial \Omega \times (0,T]$ in $\eqref{main-LF-3}_3$. 
	\end{proof}
	
	\section*{Acknowledgments}
	Li's research is partially supported by the National Natural Science Foundation of China grants 11931010, 12226326, and the key research project of Academy for Multidisciplinary Studies, Capital Normal University. Wang's research was supported by the Grant No. 830018 from China.

\end{document}